\documentclass{amsart}

\usepackage{mathtools}
\usepackage{amsmath}
\usepackage{amssymb}
\usepackage{amsthm}
\usepackage{enumitem}
\usepackage{lastpage}
\usepackage{fancyhdr}
\usepackage{accents}
\usepackage{scrextend}
\usepackage{hyperref}

\newtheorem{theorem}{Theorem}[section]
\newtheorem{lemma}[theorem]{Lemma}
\newtheorem{corollary}[theorem]{Corollary}
\newtheorem{proposition}[theorem]{Proposition}
\newtheorem{problem}[theorem]{Problem}

\theoremstyle{definition}

\newtheorem{remark}[theorem]{Remark}
\numberwithin{equation}{section}

\newcommand{\Z}{\mathbb{Z}}

\newcommand{\R}{\mathbb{R}}
\newcommand{\Q}{\mathbb{Q}}
\newcommand{\C}{\mathbb{C}}
\newcommand{\CP}{\mathbb{C}P}
\newcommand{\K}{\widetilde{K}}
\newcommand{\KO}{\widetilde{KO}}

\DeclareMathOperator{\im}{im}

\DeclareMathOperator{\rk}{rank}
\DeclareMathOperator{\id}{id}
\DeclareMathOperator{\ch}{ch}

\DeclareMathOperator{\Vect}{Vect}

\begin{document}

\title{Almost Complex Structures on Homotopy Complex Projective Spaces}

\author{Keith Mills}
\address{}
\curraddr{}
\email{kmills96@umd.edu}


\date{January 26, 2023}


\begin{abstract}
We show that all homotopy $\CP^n$s, smooth closed manifolds with the oriented homotopy type of $\CP^n$, admit almost complex structures for $3 \leq n \leq 6$, and classify these structures by their Chern classes. Our methods provide a new proof of a result of Libgober and Wood on the classification of almost complex structures on homotopy $\CP^4$s.
\end{abstract}

\maketitle

\section{Introduction} \label{introduction}

Let $X$ be a (compact orientable) smooth manifold. A \textit{complex structure} on a real vector bundle $E$ over $X$ is a complex vector bundle $F$ over $X$ whose underlying real bundle is isomorphic to $E$, and a \textit{stable complex structure} on $E$ is a complex structure on $E \oplus \epsilon^k$ for some nonnegative integer $k$, where $\epsilon^k$ is the trivial rank-$k$ bundle over $X$. We say $X$ is \textit{almost complex} if its tangent bundle $TX$ admits a complex structure, and \textit{stably almost complex} if $TX$ admits a stable complex structure.

In general, admission of almost complex structures is not a homotopy invariant of smooth closed manifolds. Kahn shows in Corollary 6 of \cite{kahn1969obstructions} how to find examples of manifolds of dimension $8k$ for arbitrary $k \geq 1$ with the same oriented homotopy type such that one is almost complex while the other is not. Further, in \cite{sutherland193example}, Sutherland constructs a pair of oriented homotopy equivalent manifolds of large dimension, one of which is a $\pi$-manifold while the other is not even stably almost complex. Sutherland's construction can also be used to give arbitrarily highly connected examples of this phenomenon and show that admission of stable almost complex structures is not even invariant under PL-homeomorphisms.

In this paper we consider the question of whether homotopy complex projective spaces admit almost complex structures. By \textit{homotopy complex projective space} we mean a smooth closed manifold with the oriented homotopy type of $\CP^n$ for some $n$. All homotopy complex projective spaces admit stable almost complex structures, a fact we will discuss in section \ref{preliminaries}. By contrast, the question of whether there exist homotopy $\CP^n$s that do not admit almost complex structures is still open at the time of the writing of this paper.

We can quickly dispense of the cases $n \leq 3$. The problem is trivial for $n=1$. For $n=2$, Theorems 1 and 2 of \cite{doldwhitney1959bundles} classify oriented rank-$4$ bundles over $4$-dimensional CW complexes: two oriented rank-$4$ bundles over a homotopy $\CP^2$ are equivalent if and only if they have the same characteristic classes $w_2$, $e$, and $p_1$, but for $4$-manifolds all of these are invariants of oriented homotopy type. Thus an orientation-preserving homotopy equivalence $f: X \to \CP^2$ pulls back $T\CP^2$ to $TX$, so every homotopy $\CP^2$ is almost complex.

Homotopy $\CP^3$s are also almost complex. From the obstruction theory point of view developed in \cite{kahn1969obstructions}, all obstructions for a homotopy $\CP^3$ to admit a stably almost complex structure vanish, and all stably complex manifolds of dimension $6$ are almost complex by the following argument. The only obstruction for a stably almost complex $6$-manifold $X$ to admit an almost complex structure lies in $H^3(X; \pi_2(SO(6)/U(3))).$ The coefficient group is in the stable range, and Massey shows in Theorem I of \cite{massey1961obstructions} that integral Stiefel-Whitney classes are integral multiples of such obstructions. Since $X$ admits a stable almost complex structure, its second Stiefel-Whitney class admits an integral lift (namely the first Chern class of the stable almost complex structure) and hence its third integral Stiefel-Whitney class vanishes, so the obstruction in question also vanishes.

For $n=4$, the answer to our question is also yes, proved by Libgober and Wood in \cite{libgoberwood1990kahler} using a result of Heaps in \cite{heaps1970almost} concerning almost complex structures on $8$-manifolds. They also provide a classification of almost complex structures on homotopy $\CP^4$s. This result uses the classification of homotopy $\CP^n$s, carried out by Brumfiel in \cite{brumfiel1968differentiable} for $n \leq 6$, as its starting point. We aim to use Brumfiel's calculations for $n=5,6$. With this in mind, our question is

\begin{problem} \label{cpprob}
Let $X$ be a smooth closed manifold homotopy equivalent to $\CP^n$ for $n = 4, 5,$ or $6$. Use the surgery classification of homotopy $\CP^n$s for such $n$ to determine whether $X$ admits an almost complex structure.
\end{problem}

We include the case $n=4$ above because our methods will provide a second proof of Libgober and Wood's result. For each value of $n$ under consideration, we will answer Problem \ref{cpprob} affirmatively and provide a classification of almost complex structures on homotopy $\CP^n$s.

The paper is structured as follows. In section \ref{preliminaries} we discuss the main tools relevant to solving Problem \ref{cpprob}, namely results that allow the problem to be solved by computations in characteristic classes as well as a summary of Brumfiel's surgery classification of homotopy complex projective spaces. The main theorem that we appeal to in solving our problem for $n=4$ and $n=6$ is Theorem \ref{cp46thm}, which says that almost complex structures are in one-to-one correspondence with complex vector bundles with the ``correct" Chern classes in the sense that certain combinations of Chern classes agree with the Pontrjagin classes of the manifold and the top-dimensional Chern class is the Euler class of the manifold. This theorem only applies when the real $K$-theory of the manifold has no $2$-torsion, and thus is not valid for homotopy $\CP^n$s where $n \equiv 1$ mod $4$. In this section we will also remark on what is required to approach this problem in general.

For $n=4, 6$ our method of proof relies on finding $n$-tuples of Chern classes over homotopy $\CP^n$s that satisfy the conditions in Theorem \ref{cp46thm}. In order to do so, one needs to be able to determine whether an integral $n$-tuple is indeed a Chern vector of some bundle over a homotopy $\CP^n$. This was already accomplished by Thomas in \cite{thomas1974almost}. In section \ref{chernvects} we characterize the image of the Chern character for $\CP^n$ and as a consequence obtain a result equivalent to Thomas' that makes finding Chern vectors that satisfy the conditions given in Theorem \ref{cp46thm} a problem in linear algebra and elementary number theory, stated in Corollary \ref{linalgforCP}.

In section \ref{accp4s} we carry out the solution to such a problem, obtaining a new proof of Libgober and Wood's result that all homotopy $\CP^4$s are almost complex. In section \ref{accp6s} we follow the same process for homotopy $\CP^6$s and find that all homotopy $\CP^6$s are almost complex.

Since Theorem \ref{cp46thm} does not apply when $n=5$, we must take a different approach for homotopy $\CP^5$s. Instead our goal is to appeal to Theorem \ref{stabunstab}, which says that one can find an almost complex structure on a homotopy $\CP^5$ provided that one can find a stable almost complex structure whose top-dimensional Chern class equals the Euler class of the manifold. As discussed in section \ref{preliminaries}, this involves finding a preimage of the class of the stable tangent bundle in $\KO(X)$ under the real reduction map $r$.

In section \ref{accp5s} we compute the real reduction map $r$ using Adams operations and a lemma of Sanderson from \cite{sanderson1964immersions}. For an arbitrary homotopy $\CP^5$ we use this computation to find such a preimage with the correct top-dimensional Chern class, showing that all homotopy $\CP^5$s admit almost complex structures. In doing so, we find a classification of stable and unstable almost complex structures on homotopy $\CP^5$s.

\section*{Acknowledgments} \label{acknowledgments}

I would like to thank Jonathan Rosenberg for his assistance and encouragement at all stages in the preparation of this paper. I also wish to thank the reviewers for helpful comments and corrections on earlier versions of the paper, alerting me to the result in \cite{doldwhitney1959bundles}, and for comments that are essentially Remark \ref{generalsetup}.

\section{Preliminaries} \label{preliminaries}

For a compact Hausdorff space $X$ denote by $K(X)$ and $KO(X)$ the complex and real $K$-groups (respectively), and $\K(X)$ and $\KO(X)$ the reduced $K$-groups. Let $r: K(X) \to KO(X)$ denote the real reduction map, and note that $r$ restricts to a map $\K(X) \to \KO(X)$ which we also denote by $r$. Complexification of real vector bundles and complex conjugation yield maps $c: KO(X) \to K(X)$ and $t: K(X) \to K(X)$, respectively, and these maps also have reduced versions which we again denote by $c$ and $t$. Note that $c$ and $t$ are ring homomorphisms while $r$ is only a homomorphism of $KO$-modules, and recall the well-known identities $r \circ c = 2: KO(X) \to KO(X)$ and $c \circ r = 1 + t: K(X) \to K(X)$.

Let $X$ be a compact orientable smooth manifold. From the point of view of $K$-theory, we can interpret the condition for a vector bundle $E$ to admit a stable complex structure as requiring that $E - \rk{E} \in \widetilde{KO}(X)$ is in the image of $r: \K(X) \to \KO(X)$.

As mentioned in section \ref{introduction}, all homotopy complex projective spaces admit stably almost complex structures. This is known, but we are unable to produce a reference for this fact.

\begin{proposition} \label{CPsareSAC}
All homotopy complex projective spaces are stably almost complex.
\end{proposition}
\begin{proof}
We will show that the real reduction map $r: \K(\CP^n) \to \KO(\CP^n)$ is surjective. A portion of the exact sequence relating real and complex $K$-theory (\cite{karoubi2008}, 5.18) reads
\[
\K^{-2}(\CP^n) \xrightarrow{r\beta^{-1}} \KO^0(\CP^n) \xrightarrow{} \KO^{-1}(\CP^n)
\]
where $\beta$ is the Bott isomorphism. But $\KO^{-1}(\CP^n) = 0$ by Fujii's computation of the $KO$-groups of complex projective space (\cite{fujii1967projective}, Theorem 2), so $r\beta^{-1}$ is surjective. It follows that $r$ is surjective.
\end{proof}

Under mild hypotheses, the obstruction to improving a stable complex structure to a genuine complex structure is the top-dimensional Chern class, as in the following theorem.

\begin{theorem}[\cite{thomas1967complex}, Theorem 1.7] \label{stabunstab}
Let $\xi$ be a real orientable rank-$2n$ bundle over a CW complex $X$. Suppose that $n>0$, $\dim X \leq 2n$, and $H^{2n}(X;\Z)$ has no $2$-torsion. Then $\xi$ admits a complex structure if and only if it admits a stable complex structure $\omega$ such that $c_{n}(\omega) = e(\xi)$, where $c_n$ denotes the $n$th Chern class and $e$ the Euler class.
\end{theorem}

In the situation that $\xi$ is the tangent bundle of a compact orientable manifold $X$ of dimension $2n$, the hypotheses of the above theorem are automatically satisfied, so $X$ is almost complex if and only if it is stably almost complex and the stably almost complex structure $E$ satisfies $c_n(E) = e(X)$. This conclusion is also found in Corollary 3 of \cite{kahn1969obstructions}.

Theorem \ref{stabunstab} will allow us to solve Problem \ref{cpprob} for $n=5$ after determining the homomorphism $r$. For $n=4,6$ we have a slightly upgraded form of the theorem:

\begin{theorem}[\cite{thomas1974almost}, Corollary 4.2] \label{cp46thm}
If $X$ is a $2n$-dimensional oriented manifold such that
\begin{enumerate}
    \item $H^{2j}(X)$ has only torsion relatively prime to $(j-1)!$,
    \item $\KO(X)$ has no $2$-torsion,
\end{enumerate}
then there is a one-to-one correspondence between almost complex structures on $X$ and elements $E \in \Vect_n(X)$, the rank-$n$ complex vector bundles over $X$, such that
\begin{enumerate}[label=(\alph*)]
    \item $c_{2i}(E \oplus \Bar{E}) = (-1)^{i}p_i(X), 1 \leq i \leq n$,
    \item $e(E) = e(X)$.
\end{enumerate}
\end{theorem}

We also review the classification of homotopy $\CP^n$s for $n=4,5,6$ as computed by Brumfiel in \cite{brumfiel1968differentiable}. Fix notation as follows. For $X = \CP^n$, one has $H^*(X;\Z) \cong \Z[u]/(u^{n+1})$, where $u$ generates $H^2(X;\Z)$. Let $H$ denote the dual of the canonical line bundle over $\CP^n$ so that its first Chern class is $c_1(H) = u$, and let $L = H - 1 \in \K(X)$ and $\omega = r(L)$. Then $K(X) \cong \Z[L]/(L^{n+1})$, and Fujii computed $KO(X)$ in Theorem 2 of \cite{fujii1967projective} as $\Z[\omega]/I$ where $I$ is the ideal $(\omega^{t+1})$ if $n=2t$, $(2\omega^{2t+1}, \omega^{2t+2})$ if $n=4t+1$, and $(\omega^{2t+2})$ if $n=4t+3$. In particular, $KO(\CP^4) = \Z[\omega]/(\omega^3)$, $KO(\CP^5) = \Z[\omega]/(2\omega^3, \omega^4)$, and $KO(\CP^6) = \Z[\omega]/(\omega^4)$.

A combination of 8.2, 8.6, 8.8, 8.10, and 8.11 of \cite{brumfiel1968differentiable} gives

\begin{theorem} \label{brumfielclassif}
Let $X$ be a smooth manifold with the oriented homotopy type of $\CP^n$, $n=4,5,6$. Then $TX = T{\CP^n} + \xi$ in $KO(\CP^n)$, where $\xi$ is a linear combination of generators of $\im{([\CP^n, G/O] \to [\CP^n, BSO])}$ given as follows:

\begin{enumerate}[label=(\roman*)]
    \item If $n=4$, $\xi = m\xi_1 + n\xi_2$ with $4m^2 - 10m - 28n = 0$, where $\xi_1 = 24\omega + 98\omega^2$ and $\xi_2 = 240\omega^2$.
    \item If $n=5$, $\xi = m\xi_1 + n\xi_2$ with $m$ even, where $\xi_1 = 24\omega + 98\omega^2 + \omega^3$ and $\xi_2 = 240\omega^2$.
    \item If $n=6$, $\xi = m\xi_1 + n\xi_2 + q\xi_3$ with $32m^3 - 252m^2 + 301m - 672mn + 1152n + 1488q = 0$, where $\xi_1 = 24\omega + 98\omega^2 + 111\omega^3$, $\xi_2 = 240\omega^2 + 380\omega^3$, and $\xi_3 = 504\omega^3$.
\end{enumerate}
\end{theorem}
It is important to note that the above statement does not completely classify homotopy $\CP^n$s, but it does classify homotopy $\CP^n$s up to torsion in the homotopy structure set of $\CP^n$. The concluding remarks of \cite{brumfiel1968differentiable} explain that these torsion elements correspond to stable tangential homotopy smoothings of $\CP^n$. Thus in light of Theorem \ref{stabunstab} the above information is all that is relevant for solving Problem \ref{cpprob}.

\begin{remark} \label{generalsetup}
We remark on what is required to approach this problem in general for an oriented simply connected $2n$-dimensional manifold $X$. In particular, a complete diffeomorphism classification is not required; one need only consider the $K$-theory classes realized by homotopy equivalences from smooth manifolds in the following sense.

Let $\sigma: [X, G/O] \to L_{2n}(e)$ denote the surgery obstruction (where $\id_X: X \to X$ is used as a basepoint to identify the set of normal invariants with $[X, G/O]$). Let $i: G/O \to BO$ denote the inclusion of the fiber of the natural map $BO \to BG$. Consider the set of stable virtual bundles
\[
\Xi = i_*(\sigma^{-1}(0)) + [TX] \in \KO(X).
\]

Then by Theorem \ref{stabunstab}, the question of which manifolds homotopy equivalent to $X$ admit almost complex structures is equivalent to determining the preimage of $e(TX)$, the Euler class of $X$, under the map
\[
c_n|_{r^{-1}(\Xi)}: r^{-1}(\Xi) \to H^{2n}(X;\Z). 
\]

Our goal is to answer this question for $\CP^n$, $4 \leq n \leq 6$, using Brumfiel's computations.
\end{remark}

\section{Chern Vectors} \label{chernvects}

In order to make use of Theorem \ref{cp46thm} we will need to know all possible Chern vectors of bundles over $\CP^n$. This is already known by Theorem A of \cite{thomas1974almost}, but we will obtain an equivalent result that makes finding solutions to conditions (a) and (b) of Theorem \ref{cp46thm} a linear algebra problem.

We begin by characterizing the image of the Chern character $\ch: K(\CP^n) \to H^*(\CP^n; \Q)$. Up to stable equivalence, every bundle over $\CP^n$ splits as a direct sum of line bundles, and $\ch(E) = \exp{(c_1(E))}$ for a line bundle $E$. Since $H^*(\CP^n;\Z)$ is torsion-free, $\im(\ch)$ is a sublattice of $\Z \oplus \Z u \oplus \Z \frac{u^2}{2} \oplus \cdots \oplus \Z \frac{u^n}{n!} \subseteq \Q^{n+1}$ generated by $\Z$-linear combinations of $q_m = 1 + mu + m^2\frac{u^2}{2} + \cdots + m^n\frac{u^n}{n!} = \ch(L_m)$ where $L_m$ is a complex line bundle over $\CP^n$ with first Chern class $mu$. We will prove the following.

\begin{theorem} \label{chernchar}
$\{q_0, q_1, \ldots, q_n\}$ forms a basis for $\im(\ch)$.
\end{theorem}
\begin{proof}
Let $\beta = \{q_0, q_1, \ldots, q_n\}$. The sum $a_0q_0 + a_1q_1 + \cdots + a_{n}q_n$ equals $(a_0 + a_1 + \cdots + a_{n}) + (a_1 + 2a_2 + 3a_3 + \cdots + na_{n})u + (a_1+ 4a_2+9a_3 + \cdots + n^2a_{n})\frac{u^2}{2} + \cdots + (a_1 + 2^na_2 + 3^na_3 + \cdots + n^na_n)\frac{u^n}{n!}$.

Checking the linear independence and spanning conditions amounts to showing that the system of equations
\[
    \begin{cases}
    a_0 + a_1 + a_2 + \cdots + a_n = 1 \\
    a_1 + 2a_2 + 3a_3 + \cdots + na_n = m \\
    a_1 + 4a_2 + 9a_3 + \cdots + n^2a_n = m^2 \\
    \hspace{45mm} \vdots \\
    a_1 + 2^na_2 + 3^na_3 + \cdots + n^na_n = m^n \\
    \end{cases}
\]
has a unique solution for an arbitrary integer $m$.

We can write this as a matrix equation $Wa=b$ where $a$ is the column vector consisting of the $a_i$s, $b=b(m)=\begin{pmatrix} 1 \\ m \\ m^2 \\ \vdots \\ m^n \end{pmatrix}$, and $W$ is the matrix
\[
W = \begin{pmatrix}
1 & 1 & 1 & 1 & \cdots & 1 \\
0 & 1 & 2 & 3 & \cdots & n \\
0 & 1 & 4 & 9 & \cdots & n^2 \\
\vdots & & & & \ddots & \vdots\\
0 & 1 & 2^n & 3^n & \cdots & n^n
\end{pmatrix}.
\]

$W$ is a (transpose of a) Vandermonde matrix and hence has determinant $\det W = \Pi_{j=1}^n j!$, so $W$ is invertible over $\Q$ and hence $\ker W = 0$. Thus it remains to verify that $W^{-1}b$ has integers in each entry independent of the value of $m$.

In order to do computations with $n \times n$ matrices instead of $(n+1) \times (n+1)$ matrices, replace $n$ by $n-1$. Then $W = V^t$ where

\[
V = \begin{pmatrix}
    1 & 0 & 0 & \cdots & 0 \\
    1 & 1 & 1 & \cdots & 1 \\
    1 & 2 & 4 & \cdots & 2^{n-1} \\
    \vdots & \vdots & \vdots & \ddots & \vdots \\
    1 & n-1 & (n-1)^2 & \cdots & (n-1)^{n-1} \\
    \end{pmatrix} =
    \begin{pmatrix}
    1 & x_1 & x_1^2 & \cdots & x_1^{n-1} \\
    1 & x_2 & x_2^2 & \cdots & x_2^{n-1} \\
    1 & x_3 & x_3^2 & \cdots & x_3^{n-1} \\
    \vdots & \vdots & \vdots & \ddots & \vdots \\
    1 & x_n & x_n^2 & \cdots & x_n^{n-1} \\
    \end{pmatrix}
\]
with $x_i = i-1$ for $1 \leq i \leq n$. The formula for the inverse of a Vandermonde matrix given in \cite{klinger1967vandermonde} gives $V^{-1} = (c_{ij})$ with

$$c_{ij} = \frac{(-1)^{n-i}\sigma_{n-i}(x_1,\ldots,\widehat{x_j},\ldots,x_n)}{\prod_{l=1,l\neq j}^n (x_j - x_l)}$$

where $\sigma_q$ is the $q$th elementary symmetric polynomial $$\sigma_q(y_1,\ldots,y_k) = \sum_{1 \leq j_1 < j_2 < \cdots < j_q \leq k}y_{j_1}y_{j_2}\cdots y_{j_q} \text{ for } q=0,1,\ldots,k.$$

Then $W^{-1} = (V^{-1})^t$, and the integrality of entries in $W^{-1}b$ will follow from the following two lemmas.
\end{proof}

\begin{lemma} \label{chints1}
The column vector $W^{-1}b(m)$ has entry in row $k$ equal to $$w_k = \frac{(-1)^{n-k}}{(n-1)!}{n-1 \choose k-1} \prod_{j=0, j\neq k-1}^{n-1} (m-j).$$
\end{lemma}
\begin{proof}
The entry of $W^{-1}b$ in row $k$ is given by $\sum_{i=1}^n(W^{-1})_{ki}m^{i-1} = \sum_{i=1}^n c_{ik}m^{i-1}$, so we wish to show that
\[
\sum_{i=1}^n \frac{(-1)^{n-i}\sigma_{n-i}(x_1,\ldots,\widehat{x_k},\ldots,x_n)}{\prod_{l=1,l\neq k}^n (k-l)}m^{i-1} = \frac{(-1)^{n-k}}{(n-1)!}{n-1 \choose k-1} \prod_{j=0, j\neq k-1}^{n-1} (m-j).
\]
It suffices to show that the coefficient of $m^{i-1}$ on the right-hand side is equal to $\frac{(-1)^{n-i}\sigma_{n-i}(x_1,\ldots,\widehat{x_k},\ldots,x_n)}{\prod_{l=1,l\neq k}^n (k-l)}$. We have
\[
\frac{1}{\prod_{l=1,l\neq k}^n (k-l)} = \frac{(-1)^{k-n}}{(k-1)!(n-k)!} = \frac{(-1)^{n-k}}{(n-1)!}{n-1 \choose k-1},
\]
hence it remains to show that the coefficient of $m^{i-1}$ in $\prod_{j=0, j\neq k-1}^{n-1} (m-j)$ equals $(-1)^{n-i}\sigma_{n-i}(x_1,\ldots,\widehat{x_k},\ldots,x_n) = (-1)^{n-i}\sigma_{n-i}(0,1,\ldots,\widehat{k-1},\ldots,n-1)$. Indeed, $x_p  = p-1$ and
\[
\prod_{j=0, j\neq k-1}^{n-1} (m-j) = [m(m-1)\cdots(m-(k-2))][(m-k)(m-(k+1))\cdots(m-(n-1))]
\]
which is a product of $n-1$ factors; if a term in the expansion of this product has exactly $i-1$ copies of $m$ then its coefficient is a product of $n-1-(i-1)=n-i$ distinct constants from the list $0,1,\ldots,k-2,k,k+1,\ldots,n-1$, with sign easily seen to be $(-1)^{n-i}$. For each unique such choice of $n-i$ constants from this list there is a single term in the expansion with degree $i-1$, giving the desired result.
\end{proof}

\begin{lemma}\label{chints2}
$w_k$ is an integer for all $1 \leq k \leq n$.
\end{lemma}
\begin{proof}
Write
\[
\prod_{j=0, j\neq k-1}^{n-1} (m-j) = [m(m-1)\cdots(m-(k-2))][(m-k)(m-(k+1))\cdots(m-(n-1))].
\]
Consider first the case that $m>0$. Note that the first grouping $m(m-1)\cdots(m-(k-2))$ equals ${m \choose m-k+1}(k-1)!$ and the second grouping $(m-k)(m-(k+1))\cdots(m-(n-1))$ equals ${m-k \choose m-n}(n-k)!$, so we have
\begin{align*}
w_k &= (-1)^{n-k}\frac{1}{(n-1)!}{n-1 \choose k-1}{m \choose m-k+1}(k-1)!{m-k \choose m-n}(n-k)! \\
    &= (-1)^{n-k}\frac{1}{(n-1)!}\frac{(n-1)!}{(k-1)!(n-k)!}{m\choose m-k+1}{m-k \choose m-n}(k-1)!(n-k)! \\
    &= (-1)^{n-k}{m \choose m-k+1}{m-k \choose m-n}
\end{align*}
which as a product of binomial coefficients is an integer.

In the case $m=0$, if $k \neq 1$ then $w_k=0$, and if $k=1$ then $\prod_{j=0, j\neq k-1}^{n-1} (m-j) = \prod_{j=1}^{n-1} (-j) = (-1)^{n-1}(n-1)!$, hence $w_1 \in \Z$ if $m=0$.

In case $m < 0$, binomial coefficients with negative entries do not quite make sense, but we can handle the product similarly. Write $m=-q$ for $q>0$, so that $\prod_{j=0, j\neq k-1}^{n-1} (-q-j)$ equals
\begin{align*}
& [(-q)(-q-1)\cdots(-q-(k-2))][(-q-k)(-q-(k+1))\cdots(-q-(n-1))] \\
    &= (-1)^{n-1}[(q+k-2)\cdots(q+1)(q)][(q+n-1)\cdots(q+k+1)(q+k)] \\
    &= (-1)^{n-1}{q+k-2 \choose q-1}(k-1)!{q+n-1 \choose q+k-1}(n-k)!,
\end{align*}
which gives
\begin{align*}
    w_k &= (-1)^{n-k}\frac{1}{(n-1)!}{n-1 \choose k-1}(-1)^{n-1}(k-1)!(n-k)!{q+k-2 \choose q-1}{q+n-1 \choose q+k-1} \\
        &= (-1)^{2n-k-1}{q+k-2 \choose q-1}{q+n-1 \choose q+k-1}
\end{align*}
which is again an integer.
\end{proof}

If we consider an arbitrary $E \in \K(X)$, we can write $\ch(E) = \Sigma_{i\geq 1}\ch_{2i}(E)u^i$ so that $\ch_{2i}(E)$ is the coefficient of $u^i$ in $\ch(E)$. Let $C(E)$ be the column vector whose entry in the $i$th position is $i!\ch_{2i}(E)$. By Theorem \ref{chernchar}, there are integers $(a_1, \ldots, a_n)$ such that $\ch(E) = \Sigma_{k \geq 1} a_kq_k$. Equating these two expressions of $\ch(E)$, we obtain $C(E) = Qa$ where $a$ is the column vector consisting of the $a_k$ and $Q$ is the $n \times n$ matrix with $Q_{ij} = j^i$ ($Q$ is essentially the matrix $W$ from the proof of Theorem \ref{chernchar} with the degree-zero portion ignored).

We can use this to characterize Chern vectors of stable bundles over $\CP^n$ as follows. By the splitting principle, $i!\ch_{2i}(E)u^i$ is a $\Z$-linear combination of elementary symmetric polynomials in the Chern classes $c_1(E), \ldots, c_i(E)$. Thus the entries of $C(E)$ can be expressed in terms of the coefficients of the Chern classes; write $C(c_1, \ldots, c_n)$ for the expression of $C(E)$ in this way. If we are given an $n$-tuple of integers $(c_1, \ldots, c_n)$, these are a Chern vector (the coefficients of Chern classes $c_k(E) = c_k u^k$) of some element $E \in \K(X)$ if and only if $C(c_1, \ldots, c_n) = Qa$ is solvable for $a$; that is, if and only if $Q^{-1}C(c_1, \ldots, c_n)$ consists solely of integer entries. This observation together with Theorem \ref{stabunstab} amounts to

\begin{corollary} \label{linalgforCP}
An $n$-tuple $(c_1, \ldots, c_n)$ is the Chern vector of a complex structure on an oriented real vector bundle $\xi$ of rank $2n$ over $\CP^n$ if and only if $Q^{-1}C(c_1, \ldots, c_n)$ lies in $\Z^n$ and $c_n u^n = e(\xi).$
\end{corollary}

To solve our problem of interest we will be setting $n=4$ and $n=6$, so we record the specific matrices $Q$ and $C$ below.

When $n=4$, we have

\begin{align*}
    C(c_1, \ldots, c_4) = \begin{pmatrix}
c_1 \\
c_1^2 - 2c_2 \\
c_1^3 - 3c_1c_2 + 3c_3 \\
c_1^4 - 4c_1^2c_2 + 4c_1c_3 + 2c_2^2 - 4c_4
\end{pmatrix},
&&
Q = \begin{pmatrix}
1 & 2 & 3 & 4 \\
1 & 2^2 & 3^2 & 4^2 \\
1 & 2^3 & 3^3 & 4^3 \\
1 & 2^4 & 3^4 & 4^4
\end{pmatrix},
\end{align*}

and when $n=6$, $C(c_1, \ldots, c_6)$ is given by
\[
\begin{pmatrix}
c_1 \\
c_1^2 - 2c_2 \\
c_1^3 - 3c_1c_2 + 3c_3 \\
c_1^4 - 4c_1^2c_2 + 4c_1c_3 + 2c_2^2 - 4c_4 \\
c_1^5 + 5c_1c_2^2 + 5c_1^2c_3 - 5c_1^3c_2 - 5c_1c_4 - 5c_2c_3 + 5c_5 \\
\scriptstyle{c_1^6 - 2c_2^3 + 3c_3^2 + 6c_1c_5 - 6c_1^2c_4 + 6c_1^3c_3 - 6c_1^4c_2 + 9c_1^2c_2^2 - 12c_1c_2c_3 + 6c_2c_4 - 6c_6}
\end{pmatrix}
\]

and

\[
Q = \begin{pmatrix}
1 & 2 & 3 & 4 & 5 & 6 \\
1 & 2^2 & 3^2 & 4^2 & 5^2 & 6^2 \\
\vdots & \vdots & \vdots & \vdots & \vdots & \vdots \\
1 & 2^6 & 3^6 & 4^6 & 5^6 & 6^6
\end{pmatrix}.
\]

\section{Almost Complex Structures on Homotopy $\CP^4$s} \label{accp4s}

Here we obtain a new proof of a result of Libgober and Wood characterizing almost complex structures on homotopy $\CP^4$s. The result is

\begin{theorem}[\cite{libgoberwood1990kahler}, Theorem 7.1] \label{libwoodthm}
Each smooth manifold $X$ homotopy equivalent to $\CP^4$
supports a (nonzero) finite number of almost complex structures. The Pontrjagin class $p_1(X) = (5 + 24m)u^2$ for some $m \equiv 0 \text{ or } 6 \pmod{14}$. Almost complex structures on $X$ correspond to integers $a$ dividing $25 + \frac{3}{7}(24^2 m^2 + 10\cdot 24m)$ under the correspondence $c_1(X) = au$.
\end{theorem}

We will prove this theorem using the classification of smooth homotopy $\CP^4s$ together with Theorem \ref{cp46thm} and Corollary \ref{linalgforCP}. By Theorem \ref{brumfielclassif}, such a manifold $X$ equals $X_{m,n}$\footnote{This is a slight, but harmless, abuse of notation: by the remarks following Theorem \ref{brumfielclassif}, $X_{m,n}$ should refer to a family of finitely many manifolds homotopy equivalent to $\CP^4$, but all of these have the same class of the stable tangent bundle in real $K$-theory. We will repeat this same abuse of notation in sections \ref{accp6s} and \ref{accp5s}.} where $m,n$ are integers such that $4m^2 - 10m = 28n$. This implies that $m \equiv 0,6 \pmod{14}$ and gives the Pontrjagin classes
\begin{align*}
&p_1(X) = (5 + 24m)u^2,\\
&p_2(X) = (10 + \frac{1}{7}(24^2 m^2 + 10\cdot24m))u^4.
\end{align*}

Condition (a) of Theorem \ref{cp46thm} holds for a bundle $E \in \Vect_4(X)$ if and only if the Chern classes of $E$ satisfy $c_2(E) = (c_1(E)^2 - p_1(X))/2$ and $2c_1(E)c_3(E) = c_2(E)^2 + 10u^4 - p_2(X)$. If we let $a = c_1$ be the coefficient of $c_1(E)$, then these formulas allow us to determine $c_2(E)$ and $c_3(E)$ from $a$ and $m$. Doing so and setting $c_4 = 5$ we have $C = C(a,c_2,c_3,5)$ given by
\[
C = \begin{pmatrix}
a \\
24m + 5 \\
(- 7a^4 + 1008a^2m + 210a^2 + 5184m^2 + 2160m + 525)/56a \\
2880m^2/7 + 1200m/7 + 5
\end{pmatrix}.
\]

Thus Corollary \ref{linalgforCP} implies that almost complex structures on $X$ are in one-to-one correspondence with values of $a$ such that $Q^{-1}C$ consists of integer entries. These values of $a$ are precisely those stated in Theorem \ref{libwoodthm}:

\begin{proposition} \label{cp4polysprop}
$Q^{-1}C$ consists of integers if and only if $a$ divides $25 + \frac{3}{7}(24^2 m^2 + 10\cdot 24m)$.
\end{proposition}
\begin{proof}
$Q^{-1}C$ is given by the column vector
\[
\begin{pmatrix}
\scriptstyle{(- 21a^4 + 3024a^2m + 1078a^2 - 7680am^2 - 14848am - 2520a + 15552m^2 + 6480m + 1575)/112a} \\
        \scriptstyle{(7a^4 - 1008a^2m - 294a^2 + 2880am^2 + 4392am + 700a - 5184m^2 - 2160m - 525)/28a}\\
\scriptstyle{(- 49a^4 + 7056a^2m + 1918a^2 - 23040am^2 - 28416am - 4200a + 36288m^2 + 15120m + 3675)/336a} \\
       \scriptstyle{(7a^4 - 1008a^2m - 266a^2 + 3840am^2 + 4064am + 560a - 5184m^2 - 2160m - 525)/224a}
\end{pmatrix}.
\]

Let $f_i$ be the numerator in row $i$. Note that the last three terms in each numerator are the only three terms not containing $a$, and that in each row these three terms are all multiples of the three terms in row 2, given by $f = -5184m^2 - 2160m - 525 = -3(1728m^2 + 720m + 175)$ (specifically, in row 1 these terms are $-3f$ and in row 3 they are $-7f$). Thus we can write
\[
Q^{-1}C = \begin{pmatrix}
    -3f/112a + (f_1 - (-3f))/112a \\
    f/28a + (f_2 - f)/28a \\
    -7f/336a + (f_3 - (-7f))/336a \\
    f/224a + (f_4 - f)/224a
\end{pmatrix}.
\]

Now, since $a$ must be odd ($X \simeq \CP^4$ is non-spin), each $f_i$ is always divisible by the relevant power of 2 that shows up in the corresponding denominator. In other words, $f_i/2^{d_i}$ is an integer for all $i$ where $d_i$ is the number of $2$s in the prime factorization of the denominator in row $i$. Hence $Q^{-1}C$ consists of integers if and only if
\[
\begin{pmatrix}
-3f/7a + (f_1 - (-3f))/7a \\
    f/7a + (f_2 - f)/7a \\
    -7f/21a + (f_3 - (-7f))/21a \\
    f/7a + (f_4 - f)/7a
\end{pmatrix}
\]
consists of integers. Now, note that $f \equiv 0$ mod $3$ and $f \equiv -4m^2 - 4m \pmod{7}$. Since $m \equiv 0,6 \pmod{14}$, we have $m^2 \equiv 0,1 \pmod{7}$, so that $f \equiv 0 \pmod{7}$. A similar computation with the other terms implies that $f_1 + 3f, f_2 - f,$ and $f_3 + 7f$ are all zero mod $7$. Further, $f_3 + 7f \equiv a^2 - a^4 \equiv 0$ mod $3$ for all $a$ by Fermat's little theorem.

Therefore, $Q^{-1}C$ consists of integers if and only if $a$ divides $f/21 = -3(1728m^2 + 720m + 175)/21 = -\frac{1}{7}(1728m^2 + 720m + 175) = -25 -\frac{3}{7}(24^2 m^2 + 10\cdot 24m).$
\end{proof}

\section{Almost Complex Structures on Homotopy $\CP^6$s} \label{accp6s}

In this section we apply the same technique to homotopy $\CP^6$s as in the previous section. By Theorem \ref{brumfielclassif}, homotopy $\CP^6$s are classified (up to torsion in the homotopy structure set) by integers $m,n,q$ satisfying $0 = 32m^3 - 252m^2 + 301m - 672mn + 1152n + 1488q$. We will prove

\begin{theorem} \label{accp6thm}
Let $X = X_{m,n,q}$ be a homotopy $\CP^6$ with $m,n,q$ as above. Almost complex structures on $X$ with $c_1(X) = au$ and $c_3(X) = cu^3$ are in one-to-one correspondence with integers $a$ and $c$ satisfying the requirements
\[
    (a,c) \equiv \begin{cases} (1 \pmod{16}, \,1 \pmod 8) \\ (7 \pmod{16},\, 3 \pmod 8) \\ (9 \pmod{16},\, 5 \pmod 8) \\ (15 \pmod{16},\, 7 \pmod 8),
    \end{cases}
\]
$(a,c) \equiv (\pm1, \pm 1) \pmod{3}$, and $a$ divides $147 - 8c^2 + 4608m^3 - 53568m^2 + 106344m - 138240nm + 362880n + 483840q$.

In particular, $(a,c) = (1,1)$ always satisfies this list of requirements, so $X$ admits a nonzero number of almost complex structures.
\end{theorem}

We follow the same steps as the proof for Theorem \ref{libwoodthm}. Using the notation of Theorem \ref{brumfielclassif}, let us first obtain formulas for the Pontrjagin classes of $X = X_{m,n,q}$. Since $\omega = r(H-1)$, the complexification of $\omega$ equals $H + t(H) - 2$ in $\K(X)$. Since $Ht(H) = 1$ and complexification is multiplicative on $\KO(X)$, we obtain
\begin{align*}
    &p_1(\omega) = u^2, &&p_2(\omega) = 0, &&p_3(\omega) = 0, \\
    &p_1(\omega^2) = 0, &&p_2(\omega^2) = -6u^4, &&p_3(\omega^2) = 20u^6, \\
    &p_1(\omega^3) = 0, &&p_2(\omega^3) = 0, &&p_3(\omega^3) = 120u^6.
\end{align*}

Since $H^*(X)$ has no $2$-torsion, the formula $p(E+F) = p(E)p(F)$ implies
\begin{align*}
p(\xi) &= p((24m)\omega + (98m + 240n)\omega^2 + (111m + 380n + 504q)\omega^3) \\
       &= p(\omega)^{24m}p(\omega^2)^{98m+240n}p(\omega^3)^{111m + 380n + 504q} \\
       &= (1+u^2)^{24m}(1 - 6u^4 + 20u^6)^{98m+240n}(1 + 120u^6)^{111m + 380n + 504q} \\
       &= 1 + (24m)u^2 + (12m(24m-1) -6(98m+240n))u^4 + (4m(24m-1)(24m-2) \\
       &- 6(24m)(98m + 240n) + 20(98m + 240n) + 120(111m + 380n + 504q))u^6.
\end{align*}

This gives
\begin{align*}
    &p_1(\xi) = (24m)u^2, \\
    &p_2(\xi) = (288m^2 - 600m - 1440n)u^4, \\
    &p_3(\xi) = (2304m^3 - 14400m^2 + 15288m + 50400n - 34560mn + 60480q)u^6,
\end{align*}

and since $p(X) = p(\xi + T{\CP^6}) = p(\xi)(1 + 7u^2 + 21u^4 + 35u^6)$, formulas for the Pontrjagin classes of $X_{m,n,q}$ are
\begin{align*}
    &p_1(X) = (7 + 24m)u^2 \\
    &p_2(X) = (21 + 288m^2 - 432m - 1440n)u^4 \\
    &p_3(X) = (35 + 2304m^3 - 12384m^2 + 11592m - 34560mn + 40320n + 60480q)u^6.
\end{align*}

Let $E \in \Vect_6(X)$ be an arbitrary bundle and write $c_i(E) = a_iu^i$ and $p_i(X) = p_iu^{2i}$. Then the conditions of Theorem \ref{cp46thm} holds for $E$ if and only if the coefficients of the Chern classes of $E$ satisfy $a_6=7$ and
\begin{align*}
    &a_2 = \frac{a_1^2 - p_1}{2} \\
    &a_4 = a_1a_3 + \frac{p_2 - a_2^2}{2} \\
    &a_5 = \frac{7 + a_2a_4}{a_1} + \frac{p_3 - a_3^2}{2a_1}.
\end{align*}

Thus the values of $a_2, a_4,$ and $a_6$ are determined from $p_1, p_2, p_3, a = a_1$, and $c = a_3$. Using these formulas together with the expression for $C$ found at the end of section \ref{chernvects} and the Pontrjagin classes above, one obtains that $C = C(a_1,\ldots,a_5,7)$ is given by
\[
C = \begin{pmatrix} a \\
                    7 + 24m \\
                    21a/2 + 3c + 36am - a^3/2 \\
                    7 + 1200m + 2800n \\
                    \scriptstyle{a^5/16 - 15ma^3/2 - 35a^3/16 + 315a/16 + 180m^2a + 1800na + 855ma + 5\varphi/(16a)} \\
                    7 + 45864m + 151200n + 181440q
\end{pmatrix},
\]
where $\varphi = 147 + 4608m^3 - 53568m^2 + 106344m - 138240nm + 362880n + 483840q - 8c^2$.

By Corollary \ref{linalgforCP}, almost complex structures on $X$ are in one-to-one correspondence with values of $a=a_1$ and $c=a_3$ such that $Q^{-1}C$ consists of integer entries. Using $Q$ given at the end of section \ref{chernvects}, one obtains that $v$ is a $6 \times 1$ column vector where the entry in row $i$ is a rational expression whose numerator is a polynomial $f_i = f_i(a,c,m,n,q)$ and whose denominator is either $96a, 768a,$ or $120a$. Specifically, $v$ is\footnote{This is possible, but tedious, to obtain by hand. In the words that follow this and in the proof of Proposition \ref{cp6polysprop} we make a number of claims that are also impractical to verify by hand. We used MATLAB to obtain and verify these claims; the code used that justifies our assertions is available upon request.}
\[ v = 
\begin{pmatrix}
f_1/96a \\
f_2/768a \\
f_3/96a \\
f_4/768a \\
f_5/120a \\
f_6/768a
\end{pmatrix}
\]
with
\begin{align*}
    f_1 &= 5\varphi - 6720a + 1392ac - 205536am - 478080an - 145152aq + 30384a^2m \\ &- 120a^4m + 28800a^2n + 5763a^2 - 267a^4 + a^6 + 2880a^2m^2 \\
    f_2 &= -95\varphi + 94080a - 22128ac + 3633792am + 8732160an + 2903040aq \\ &- 525456a^2m + 2280a^4m - 547200a^2n - 89193a^2 + 4353a^4 - 19a^6 - 54720a^2m^2
    \\
    f_3 &= 15\varphi - 11760a + 2976ac - 542016am - 1332480an - 483840aq \\ &+ 76752a^2m - 360a^4m + 86400a^2n + 12001a^2 - 601a^4 + 3a^6 + 8640a^2m^2
    \\
    f_4 &= -85\varphi + 56448a - 14736ac + 2940288am + 7349760an + 2903040aq \\ &- 409392a^2m + 2040a^4m - 489600a^2n - 59811a^2 + 3051a^4 - 17a^6 - 48960a^2m^2
    \\
    f_5 &= 5\varphi - 2940a + 780ac - 167640am - 424800an - 181440aq  \\ &+ 23040a^2m - 120a^4m + 28800a^2n + 3189a^2 - 165a^4 + a^6 + 2880a^2m^2
    \\
    f_6 &= - 5\varphi + 2688a - 720ac + 164736am + 422400an + 193536aq \\ &- 22320a^2m + 120a^4m - 28800a^2n - 2963a^2 + 155a^4 - a^6 - 2880a^2m^2
\end{align*}
and note that $96 = 2^5\cdot 3, 768 = 2^8 \cdot 3,$ and $120 = 2^3 \cdot 3 \cdot 5$. Thus $v$ consists of integer entries if and only if all $f_i$ are zero modulo the appropriate prime powers and divisible by $a$.

The following proposition completes the proof of Theorem \ref{accp6thm}.

\begin{proposition} \label{cp6polysprop}
$v = Q^{-1}C$ consists of integers if and only if $a$ and $c$ satisfy the requirements stated in Theorem \ref{accp6thm}.
\end{proposition}
\begin{proof}
Divisibility of the appropriate $f_i$ by $2^3$ and $2^5$ is independent of $m,n,q$ and only requires $a,c$ odd. Divisibility by $2^8$ is still independent of $m,n,q$ and requires

\[
    (a,c) \equiv \begin{cases} (1 \pmod{16}, \,1 \pmod 8) \\ (7 \pmod{16},\, 3 \pmod 8) \\ (9 \pmod{16},\, 5 \pmod 8) \\ (15 \pmod{16},\, 7 \pmod 8).
    \end{cases}
\]

Let $f$ be the collection of all terms not containing $a$ in $f_1$; note that $f = 5\varphi$.

Grouping the terms in each $f_i$ into terms that do not contain $a$ and terms that do contain $a$ gives
\[v = \begin{pmatrix}
f/96a + (f_1 - f)/96a \\
-19f/768a + (f_2 - (-19f))/768a \\
3f/96a + (f_3 - 3f)/96a \\
-17f/768a + (f_4 - (-17f))/768a \\
f/120a + (f_5 - f)/120a \\
-f/768a + (f_6 - (-f))/768a
\end{pmatrix}
\]
and under the restrictions on $a$ and $c$ mod $16$ and $8$ (respectively), $v$ consists of integers if and only if
\[
\begin{pmatrix}
f_1/3a \\ f_2/3a \\ f_3/3a \\ f_4/3a \\ f_5/15a \\ f_6/3a
\end{pmatrix}
=
\begin{pmatrix}
f/3a + (f_1 - f)/3a \\
-19f/3a + (f_2 - (-19f))/3a \\
3f/3a + (f_3 - 3f)/3a \\
-17f/3a + (f_4 - (-17f))/3a \\
f/15a + (f_5 - f)/15a \\
-f/3a + (f_6 - (-f))/3a
\end{pmatrix}
\]
consists of integers.

Concerning divisibility by $5$, we see that $f \equiv 0$ mod 5 and reducing the coefficients of $f_5$ mod 5 we get $4a^2 + a^6$, which is congruent to $0$ mod $5$ by Fermat's little theorem.

Now we claim that the above vector consists of integers if and only if $(a,c) \equiv (\pm1, \pm 1)$ mod 3. To prove this, note that $f \equiv 2c^2$ mod $3$, and reducing the coefficients of each $f_i$ modulo $3$ (and using Fermat's little theorem) yields $\pm(a^2 - c^2)$, except for $f_3$ where we obtain 0 mod 3. Thus each $f_i$ is divisible by 3 if and only if $a^2 - c^2 \equiv 0$ mod 3. Solutions to this equation are $(a,c) \equiv (0,0), (\pm1, \pm 1)$ mod 3.

In the case $(a,c) \equiv (0,0)$ mod $3$, we have $f \equiv 0$ mod $3$. To obtain integers in $v$, then, we must have a value of $a$ that divides $f/(3\cdot5) = \varphi/3 = 49 - \frac{8}{3}c^2 + 35448m + 120960n + 161280q - 46080mn - 17856m^2 + 1536m^3$. Call this $\tilde{f},$ so there is some integer $k$ such that $ka = \tilde{f}$. But if $c$ is zero mod 3, then all terms of $\tilde{f}$ reduce to zero mod 3 except for the constant term $49$, so we get $\tilde{f} \equiv 1$ mod $3$, which gives a contradiction since $a \equiv 0$ mod $3$.

This shows that $(a,c) \equiv (\pm1, \pm1)$ mod $3$ are the only possibilities; in particular, $a$ and $3$ are relatively prime and $f$ is not divisible by $3$. Therefore, under these conditions, $v$ consists of integers if and only if $a$ divides $f/5 = \varphi$.
\end{proof}

\section{Almost Complex Structures on Homotopy $\CP^5$s} \label{accp5s}

Since $\KO(\CP^5) \cong \Z[\omega]/(2\omega^3, \omega^4)$ has $2$-torsion, we cannot use the previous method to determine whether homotopy $\CP^5$s admit almost complex structures. Instead, we will find for $X \simeq \CP^5$ an element $E \in \K(X)$ whose real reduction is the stable tangent bundle of $X$ and whose top-dimensional Chern class is the Euler class of $X$, so that $X$ admits an almost complex structure by Theorem \ref{stabunstab}.

First we determine a formula for the real reduction map $r: K(X) \to KO(X)$. Consider first $r$ with $2$-torsion ignored; that is, the composite $\bar{r}: K(X) \xrightarrow{r} KO(X) \twoheadrightarrow KO(X)/(\omega^3)$.

\begin{proposition} \label{rbar}
$\bar{r}$ on the additive generators of $K(X)$ is given as follows.
\begin{align*}
    &1 \mapsto 2 && L^3 \mapsto 3\omega^2 \\
    &L \mapsto \omega && L^4 \mapsto 2\omega^2 \\
    &L^2 \mapsto 2\omega + \omega^2 && L^5 \mapsto 0
\end{align*}
\end{proposition}
\begin{proof}
$r(1) = 2$ is trivial and $r(L) = \omega$ by definition, so the same is true for $\bar{r}$. For the remaining powers, first note that since $1 = H^{-1}H = H^{-1}(1 + L)$, one has $H^{-1} = (1+L)^{-1}$ so that $t(L) = t(H-1) = H^{-1} - 1 = -L + L^2 - L^3 + L^4 - L^5$.

Since $\KO(X)$ is spanned by $\omega, \omega^2$, and $\omega^3$, we have $r(L^2) = m\omega + n\omega^2 + q\omega^3$. Applying $c$ we obtain $c \circ r (L^2) = mc(\omega) + nc(\omega^2)$. Since $c \circ r = 1 + t$ and $c(\omega) = L^2 - L^3 + L^4 - L^5$ (Lemma 3.13 of \cite{sanderson1964immersions}), this equation implies $L^2 + t(L)^2 = m(L^2 - L^3 + L^4 - L^5) + n(L^4 - 2L^5)$. The left-hand side of this equation is $2L^2 - 2L^3 + 3L^4 - 4L^5$, so $m=2$ and $n=1$.

Since $c \circ r(L^2) = 2c(\omega) + c(\omega^2) = c(2\omega + \omega^2)$ and $c$ is injective modulo $2$-torsion, we have $r(L^2) = 2\omega + \omega^2$ modulo $2$-torsion.

Similarly, one obtains $c\circ r (L^3) = c(3\omega^2), c \circ r(L^4) = c(2\omega^2), $ and $c \circ r(L^5) = 0$. Injectivity of $c$ modulo $2$-torsion gives $\bar{r}$ on these powers of $L$.
\end{proof}

To resolve the ambiguity in $2$-torsion we will make use of the Adams operations on real $K$-theory. Let $\psi_{\C}^k$ and $\psi_{\R}^k$ denote the Adams operations $\psi^k$ on complex and real $K$-theory, respectively, and omit subscripts when it is clear from context. Recall that $\psi_{\R}^k \circ r = r \circ \psi_{\C}^k$ (Proposition 7.40 of \cite{karoubi2008}) and on complex projective space, $\psi_{\C}^k(L) = (L+1)^k - 1$. Thus we have
\[
\psi^k(\omega) = \psi_{\R}^k \circ r (L) = r \circ \psi_{\C}^k(L) = r((L+1)^k - 1).
\]
\begin{proposition} \label{rL2L4}
$r(L^2) = \omega^2 + 2\omega$ and $\psi^2(\omega) = \omega^2 + 4\omega$. Together these imply $\psi^4(\omega) = 20\omega^2 + 16\omega$ and $r(L^4) = 2\omega^2$.
\end{proposition}
\begin{proof}
For any prime $p$ and element $x \in KO(X)$ we have $\psi^p(x) = x^p$ mod $p$, so $\psi^2(\omega) = \omega^2$ mod 2. On the other hand, $\psi^2(\omega) = r((L+1)^2 - 1) = r(L^2 + 2L) = r(L^2)$ mod $2$, so $r(L^2) \equiv \omega^2$ mod 2; in other words, $r(L^2) = \omega^2 + 2\alpha$ for some class $\alpha$. But Proposition \ref{rbar} implies either $r(L^2) = \omega^2 + 2\omega$ or $\omega^2 + 2\omega + \omega^3$, and only the former is consistent with our computation modulo $2$ since $2\omega^3 = 0$.

Returning to $\psi^2$, we have $\psi^2(\omega) = r(L^2) + 2r(L) = \omega^2 + 2\omega + 2\omega = \omega^2 + 4\omega$.

Now since $\psi^4 = \psi^2\psi^2$ we have $\psi^4(\omega) = \psi^2(\omega^2 + 4\omega) = \psi^2(\omega)^2 + 4\psi^2(\omega) = (\omega^2 + 4\omega)^2 + 4(\omega^2 + 4\omega) = 20\omega^2 + 16\omega$.

However, we also have $\psi^4(\omega) = r((L+1)^4 - 1) = r(L^4 + 4L^3 + 6L^2 + 4L) = r(L^4) + 4r(L^3) + 6r(L^2) + 4r(L) = r(L^4) + 4r(L^3) + 6\omega^2 + 16\omega$. Thus $r(L^4) + 4r(L^3) = 14\omega^2$. If $r(L^3)$ has $\omega^3$ as a summand, the coefficient $4$ will kill this term and we conclude $r(L^4)$ cannot contain $\omega^3$ as a summand. Thus Proposition \ref{rbar} implies $r(L^4) = 2\omega^2$.
\end{proof}

Thus $\omega^3$ does not appear in the image of $r$ on even powers of $L$. The situation is different for the remaining odd powers:
\begin{proposition} \label{rL3L5}
$r(L^3) = \omega^3 + 3\omega^2$ and $r(L^5) = \omega^3$. 
\end{proposition}
\begin{proof}
By Proposition 4.3 of \cite{thomas1974almost}, $\ker r$ is freely generated by $\mu_1 = H - H^{-1}$, $\mu_2 = H^2 - H^{-2}$, and $2L^5$. Using our computation for $H^{-1}$ and $t(L)$ from Proposition \ref{rbar}, we see that $\mu_1 = L - t(L) = 2L - L^2 + L^3 - L^4 + L^5$. Hence $0 = r(\mu_1) = 2r(L) - r(L^2) + r(L^3) - r(L^4) + r(L^5) = 2\omega - \omega^2 - 2\omega + r(L^3) - 2\omega^2 + r(L^5) = r(L^3) + r(L^5) - 3\omega^2$.

Now $r(L^3) + r(L^5) = 3\omega^2$ implies that either both $r(L^3)$ and $r(L^5)$ contain $\omega^3$ as a summand or neither do. If the latter is the case, then $\omega^3$ is not in the image of $r$, but $r$ is surjective as discussed in Proposition \ref{CPsareSAC}. Hence both $r(L^3)$ and $r(L^5)$ contain $\omega^3$ as a summand, and Proposition \ref{rbar} implies $r(L^3) = \omega^3 + 3\omega^2$ and $r(L^5) = \omega^3$.
\end{proof}

Summarizing, we have

\begin{theorem} \label{rformula}
$r: K(X) \to KO(X)$ on the additive generators of $K(X)$ is given as follows.
\begin{align*}
    &1 \mapsto 2 && L^3 \mapsto \omega^3 + 3\omega^2 \\
    &L \mapsto \omega && L^4 \mapsto 2\omega^2 \\
    &L^2 \mapsto 2\omega + \omega^2 && L^5 \mapsto \omega^3
\end{align*}
Moreover, the Adams operation $\psi^k$ on $KO(X)$ sends $\omega$ to $r((L+1)^k - 1)$; in particular, $\psi^2(\omega) = \omega^2 + 4\omega$ and $\psi^4(\omega) = 20\omega^2 + 16\omega$.
\end{theorem}

The final information necessary to verify that an element $E \in \K(X)$ has the desired top-dimensional Chern class is the Chern classes of the powers of $L = H-1$. Let $c_*$ denote the total Chern class (so as not to confuse it with the complexification homomorphism $c$). We begin with the formulas
\begin{align*}
    &L = H-1 \\
    &L^2 = (H^2 - 1) - 2L \\
    &L^3 = (H^3 - 1) - 3L^2 - 3L \\
    &L^4 = (H^4 - 1) - 4L^3 - 6L^2 - 4L \\
    &L^5 = (H^5 - 1) - 5L^4 - 10L^3 - 10L^2 - 5L.
\end{align*}
Now, since $H^k$ is a line bundle and $c_1$ is additive on products of line bundles, we have $c_*(H^k - 1) = 1 + ku$. This gives us
\begin{align*}
    &c_*(L) = 1 + u \\
    &c_*(L^2) = (1 + 2u)(1+u)^{-2} = 1 - u^2 + 2u^3 - 3u^4 + 4u^5 \\
    &c_*(L^3) = (1+3u)c_*(L^2)^{-3}(1+u)^{-3} = 1 + 2u^3 - 9u^4 + 30u^5 \\
    &c_*(L^4) = (1+4u)c_*(L^3)^{-4}c_*(L^2)^{-6}(1+u)^{-4} = 1 - 6u^4 + 48u^5 \\
    &c_*(L^5) = (1+5u)c_*(L^4)^{-5}c_*(L^3)^{-10}c_*(L^2)^{-10}(1+u)^{-5} = 1 + 24u^5.
\end{align*}

We can now prove that all homotopy $\CP^5$s admit almost complex structures. By Theorem \ref{brumfielclassif}, smooth manifolds homotopy equivalent to $\CP^5$ are given by $X = X_{m,n}$ with $TX = (24m)\omega + (98m + 240n)\omega^2 + T{\CP^5}$ in $\KO(X)$ for integers $m$ and $n$ with $m$ even. We begin by determining the set of stable almost complex structures on $X$.

\begin{lemma} \label{saccp5s}
Let $X = X_{m,n}$ be a homotopy $\CP^5$ as above. The set of stable almost complex structures on $X$ is
$$\{E = \Sigma_{i=1}^5 k_i L^i : k_1 + 2k_2 = 6 + 24m, k_2 + 3k_3 + 2k_4 = 98m + 240n, k_3 + k_5 \equiv 0 \text{ mod } 2 \}.$$
\end{lemma}
\begin{proof}
Since the holomorphic tangent bundle of $\CP^5$ equals $6H - 1$ in $K(X)$, we have $T{\CP^5} = 6\omega$ in $\KO(X)$ so that $TX = (6 + 24m)\omega + (98m + 240n)\omega^2$ in $\KO(X)$. Using Theorem \ref{rformula} one obtains, for an arbitrary element $E = \Sigma_{i=1}^5 k_iL^i \in \K(X)$, that
\[
r(E) = \Sigma_{i=1}^5 k_i r(L^i)
= (k_1 + 2k_2)\omega + (k_2 + 3k_3 + 2k_4)\omega^2 + (k_3 + k_5)\omega^3.
\]
Setting $r(E) = TX$ completes the proof.
\end{proof}

\begin{lemma}\label{topcherncp5s}
The top-dimensional Chern class of $E = \Sigma_{i=1}^5 k_i L^i \in \K(X)$ is $Ku^5$, where $K = K(k_1,\ldots,k_5)$ is given by
\begin{multline*}
    30k_3 + 48k_4 + 24k_5 + k_1(-9k_3 - 6k_4) + 2k_3\left({k_1 \choose 2} - k_2\right) \\ + 4k_2 - 4{k_2 \choose 2} + k_1\left(-3k_2 + {k_2 \choose 2}\right) + k_2\left(2{k_1 \choose 2} - {k_1 \choose 3}\right) + {k_1 \choose 5}.
\end{multline*}
\end{lemma}
\begin{proof}
The total Chern class of $E$ is $c_*(E) = \Pi_{i=1}^5 c_*(L^i)^{k_i}$. We wish to determine the coefficient of $u^5$ in $c_*(E)$. Using binomial expansions and the formulas above, one has
\begin{align*}
    &c_*(L)^{k_1} = (1 + u)^{k_1} = 1 + k_1 u + {k_1 \choose 2}u^2 + \cdots + {k_1 \choose 5}u^5 \\
    &c_*(L^2)^{k_2} =  1 - k_2u^2 + 2k_2u^3 + \left({k_2 \choose 2} - 3k_2\right)u^4 + \left(4k_2 - 4{k_2 \choose 2}\right)u^5\\
    &c_*(L^3)^{k_3} = 1 + 2k_3u^3 - 9k_3u^4 + 30k_3u^5 \\
    &c_*(L^4)^{k_4} = 1 - 6k_4u^4 + 48k_4u^5 \\
    &c_*(L^5)^{k_5} = 1 + 24k_5u^5
\end{align*}
where we use the conventions that ${i \choose j} = 0$ if $i \geq 0$ and $i < j$, and ${-i \choose j} = (-1)^j{i + j - 1 \choose j}$ for $i > 0$.

Thus $c_*(k_3L^3 + k_4L^4 + k_5L^5) = 1 + 2k_3u^3 - (9k_3 + 6k_4)u^4 + (30k_3 + 48k_4 + 24k_5)u^5$, so obtaining the coefficient of $u^5$ in $c_*(E)$ only requires knowledge of the coefficients of $u$, $u^2$, and $u^5$ in $c_*(k_1 L + k_2 L^2)$. These are $k_1$, ${k_1 \choose 2} - k_2$, and $4k_2 - 4{k_2 \choose 2} + k_1\left(-3k_2 + {k_2 \choose 2}\right) + k_2\left(2{k_1 \choose 2} - {k_1 \choose 3}\right) + {k_1 \choose 5}$ (respectively). It follows directly from this that the coefficient of $u^5$ in $c_*(E)$ is as claimed.
\end{proof}

The previous two lemmas combine to give a classification of almost complex structures on homotopy $\CP^5$s, by simply adding the requirement that $K = 6$ in the set of almost complex structures on homotopy $\CP^5$s given in Lemma \ref{saccp5s}. We claim that the resulting set is nonempty by exhibiting an explicit element:

\begin{theorem} \label{accp5thm}
Let $E = 6L + 12mL^2 + 80n L^3 + 43m L^4 + (-19m - 20n - 6m^2 + 80mn)L^5$. Then $E$ is a stable almost complex structure on $X_{m,n}$ with top Chern class $c_5(E) = e(X) = 6u^5$, hence $X$ admits an almost complex structure.
\end{theorem}
\begin{proof}
Note that $E$ satisfies the conditions of Lemma \ref{saccp5s} (the coefficient of $\omega^3$ vanishes since $m$ is even), so $E$ is a stable almost complex structure on $X$. It remains to show that $K = K(6, 12m, 80n, 43m, -19m-20n-6m^2+80mn)$ from Lemma \ref{topcherncp5s} equals 6.

$K$ has several terms so we proceed in smaller steps. First, substituting $k_1 = 6, k_2 = 12m$ into $K$, we obtain
\begin{multline*}
    30k_3 + 48k_4 + 24k_5 + 6(-9k_3 - 6k_4) + 2k_3(15 - 12m) \\ + 72m - 288m^2 + 6(72m^2 - 42m) + 12m(10) + 6,
\end{multline*}
which simplifies to $6 + 6k_3 + 12k_4 + 24k_5 + 144m^2 - 60m - 24mk_3$.
But one computes
\begin{align*}
    \tilde{K} &= k_3 + 2k_4 + 4k_5 + 24m^2 - 10m - 4mk_3 \\
    &= 80n + 2(43m) + 4(-19m - 20n - 6m^2 + 80mn) + 24m^2 - 10m - 4m(80n) \\
    &= 80n + 86m - 76m - 80n - 24m^2 + 320mn + 24m^2 - 10m - 320mn \\
    &= 0,
\end{align*}
and $K = 6 + 6\tilde{K}$, hence $K = 6$. Thus $c_5(E) = 6u^5 = e(X)$. By Theorem \ref{stabunstab}, $X$ is almost complex.
\end{proof}

\bibliography{cps}

\begin{thebibliography}{Kah69}

\bibitem[Bru69]{brumfiel1968differentiable}
G~Brumfiel.
\newblock Differentiable {$S^1$} actions on homotopy spheres, mimeographed.
\newblock {\em Princeton University}, 1969.

\bibitem[DW59]{doldwhitney1959bundles}
A.~Dold and H.~Whitney.
\newblock Classification of oriented sphere bundles over a {$4$}-complex.
\newblock {\em Ann. of Math. (2)}, 69:667--677, 1959.

\bibitem[Fuj67]{fujii1967projective}
Michikazu Fujii.
\newblock {$K_{O}$}-groups of projective spaces.
\newblock {\em Osaka Math. J.}, 4:141--149, 1967.

\bibitem[Hea70]{heaps1970almost}
Terence Heaps.
\newblock Almost complex structures on eight- and ten-dimensional manifolds.
\newblock {\em Topology}, 9:111--119, 1970.

\bibitem[Kah69]{kahn1969obstructions}
Peter~J. Kahn.
\newblock Obstructions to extending almost {$X$}-structures.
\newblock {\em Illinois J. Math.}, 13:336--357, 1969.

\bibitem[Kar08]{karoubi2008}
Max Karoubi.
\newblock {\em {$K$}-theory}.
\newblock Classics in Mathematics. Springer-Verlag, Berlin, 2008.
\newblock An introduction, Reprint of the 1978 edition, With a new postface by
  the author and a list of errata.

\bibitem[Kli67]{klinger1967vandermonde}
Allen Klinger.
\newblock The {V}andermonde matrix.
\newblock {\em Amer. Math. Monthly}, 74:571--574, 1967.

\bibitem[LW90]{libgoberwood1990kahler}
Anatoly~S. Libgober and John~W. Wood.
\newblock Uniqueness of the complex structure on {K}\"{a}hler manifolds of
  certain homotopy types.
\newblock {\em J. Differential Geom.}, 32(1):139--154, 1990.

\bibitem[Mas61]{massey1961obstructions}
W.~S. Massey.
\newblock Obstructions to the existence of almost complex structures.
\newblock {\em Bull. Amer. Math. Soc.}, 67:559--564, 1961.

\bibitem[San64]{sanderson1964immersions}
B.~J. Sanderson.
\newblock Immersions and embeddings of projective spaces.
\newblock {\em Proc. London Math. Soc. (3)}, 14:137--153, 1964.

\bibitem[Sut73]{sutherland193example}
W.~A. Sutherland.
\newblock An example concerning stably complex manifolds.
\newblock {\em J. London Math. Soc. (2)}, 6:348--350, 1973.

\bibitem[Tho67]{thomas1967complex}
Emery Thomas.
\newblock Complex structures on real vector bundles.
\newblock {\em Amer. J. Math.}, 89:887--908, 1967.

\bibitem[Tho74]{thomas1974almost}
Alan Thomas.
\newblock Almost complex structures on complex projective spaces.
\newblock {\em Trans. Amer. Math. Soc.}, 193:123--132, 1974.

\end{thebibliography}
\bibliographystyle{alpha}

\end{document}